\documentclass[11pt,reqno,a4paper]{amsart}
\usepackage{amscd}
\usepackage{amsmath}
\usepackage{amsfonts}
\usepackage{amssymb}
\usepackage{color}
\newtheorem{theorem}{Theorem}[section]
\theoremstyle{plain}

\newtheorem{cor}[theorem]{Corollary}

\newtheorem{lemma}[theorem]{Lemma}

\numberwithin{equation}{section}
\newcommand{\N}{\mathbb{N}}
\newcommand{\R}{\mathbb{R}}
\newcommand{\Z}{\mathbb{Z}}

\definecolor{red}{rgb}{1,0,.2}

\allowdisplaybreaks

\begin{document}
\title[IFS with random maps in $L^2$]
{The absolute continuity of the invariant measure of random
iterated function systems with overlaps}

\author{Bal\'azs B\'ar\'any}
\address{Bal\'azs B\'ar\'any, IM PAN, \'Sniadeckich 8, P.O. Box 21, 00-956 Warszawa 10, Poland } \email{balubsheep@gmail.com}

\author{Tomas Persson}
\address{Tomas Persson, IM PAN, \'Sniadeckich 8, P.O. Box 21, 00-956 Warszawa 10, Poland }
\email{tomasp@impan.pl}

 \thanks{
\\ \indent
{\em Key words and phrases.} Iterated function systems, Absolute continuity, random perturbation\\
\indent Research of B\'ar\'any was supported by the EU FP6
Research Training Network CODY}

\begin{abstract}
  We consider iterated function systems on the interval with random perturbation. Let $Y_\varepsilon$ be uniformly distributed in $[1- \varepsilon, 1 + \varepsilon]$ and let $f_i \in C^{1+\alpha}$ be contractions with fixpoints $a_i$. We consider the iterated function system $\{ Y_\varepsilon f_i + a_i (1 - Y_\varepsilon) \}_{i=1}^n$, were each of the maps are chosen with probability $p_i$. It is shown that the invariant density is in $L^2$ and the $L^2$-norm does not grow faster than $1/\sqrt{\varepsilon}$, as $\varepsilon$ vanishes.

  The proof relies on defining a piecewise hyperbolic dynamical system on the cube, with an SRB-measure with the property that its projection is the density of the iterated function system.
\end{abstract}
\date{\today}

\maketitle

\thispagestyle{empty}

\section{Introduction and Statements of Results}
Let $\left\{f_1,\ldots,f_l\right\}$ be an iterated function system
(IFS) on the real line, where the maps are applied according to the
probabilities $(p_1,\ldots,p_l)$, with the choice of the map random
and independent at each step. We assume that for each $i$, $f_i$
maps $[-1,1)$ into itself and $f_i\in C^{1+\alpha}([-1,1))$.
Let $\nu$ be the invariant measure of our IFS, namely,
\begin{equation}\label{e1}
\nu=\sum_{i=1}^lp_i\nu\circ f_i^{-1}.
\end{equation}
Let $\mu=(p_1,\ldots,p_l)^{\N}$ be a Bernoulli measure on the space
$\sum=\left\{1,\ldots,l\right\}^{\N}$. Let
$h(\underline{p})=-\sum_{i=1}^lp_i\log p_i$ be the entropy of the
underlying Bernoulli process $\mu$. It was proved in \cite{NSB}
for non-linear contracting on average IFSs (and later extended in
\cite{FST}) that
\[
  \dim_\mathrm{H} (\nu)\leq\frac{h}{|\chi|},
\]
where $\dim_\mathrm{H} (\nu)$ is the Hausdorff dimension of the measure
$\nu$ and $\chi$ is the Lyapunov exponent of the IFS associated to
the Bernoulli measure $\mu$.

One can expect that, at least ''typically'', the measure $\nu$ is
absolutely continuous when $h/|\chi|>1$. Essentially the only known
approach to this is transversality. For example, in linear case
with uniform contracting ratios see \cite{peschl},\cite{peso2}. In the
linear case for non-uniform contracting ratios, see \cite{Neu},
\cite{ngawang}. In the non-linear case, see for example \cite{SSU2},
\cite{BPS}. We note that there is an other direction in the study of
IFSs with overlaps, which is concerned with concrete, but
not-typical systems, often of arithmetic nature, for which there
is a dimension drop, see, for example \cite{LNR}.

Trough this paper we are interested in to study absolute
continuity with density in $L^2$. We study a modification of the
problem, namely we consider a random perturbation of the
functions. The linear case was studied by Peres, Simon and Solomyak in \cite{PSS}. They
proved absolute continuity for random linear IFS, with non-uniform
contracting ratios and also $L^2$ and continuous density in the
uniform case. We would like to extend this result by proving $L^2$
density with non-uniform contracting ratios and in non-linear
case.

We consider two cases. First let us suppose that for each
$i\in\left\{1,\ldots,l\right\}$, $f_i$ maps $[-1,1)$ into itself,
$f_i\in C^{1+\alpha}([-1,1))$ and
\begin{equation}\label{e3}
 0 < \lambda_{i,\min} \leq | f_i'(x) | \leq \lambda_{i,\max} < 1
\end{equation}
for every $x\in[-1,1)$. Moreover let us suppose that for every $i$
the fix point of $f_i$ is $a_i\in[-1,1]$, and
\begin{equation}\label{e4}
  i \neq j \ \Rightarrow \ a_i \neq a_j.
\end{equation}

Let $Y_{\varepsilon}$ be uniformly distributed on
$[1-\varepsilon,1+\varepsilon]$. Let us denote the probability
measure of $Y_{\varepsilon}$ by $\eta_{\varepsilon}$. Let
\begin{equation}
f_{i,
Y_{\varepsilon}}(x)=Y_{\varepsilon}f_i(x)+a_i(1-Y_{\varepsilon})
\end{equation}
for every $i\in\left\{1,\ldots,l\right\}$. The iterated maps are
applied randomly according to the stationary measure $\mu$, with
the sequence of independent and identically distributed errors
$y_1,y_2,\ldots$, distributed as $Y_{\varepsilon}$, independent of
the choice of the function. The Lyapunov exponent of the IFS is
defined by
\[
  \chi(\mu,\eta_{\varepsilon})
  = \mathbb{E}(\log(Y_{\varepsilon}f'))
\]
and
\[
  \chi(\mu,\eta_{\varepsilon})
  <\sum_{i=1}^lp_i\log((1+\varepsilon)\lambda_{i,\max}) < 0,
\]
for sufficiently small $\varepsilon>0$. Let $Z_{\varepsilon}$ be the
following random variable
\begin{equation}\label{e2}
Z_{\varepsilon}:=\lim_{n\rightarrow\infty}f_{i_1,y_{1,\varepsilon}}\circ
f_{i_2,y_{2,\varepsilon}}\circ\cdots\circ
f_{i_n,y_{n,\varepsilon}}(0),
\end{equation}
where the numbers $i_k$ are i.i.d., with the distribution $\mu$ on
$\left\{1,\ldots,l\right\}$, and $y_k$ are pairwise independent
with distribution of $Y_{\varepsilon}$ and also independent of the
choice of $i_k$. Let $\nu_{\varepsilon}$ be the distribution of
$Z_{\varepsilon}$.


One can easily prove the following theorem.
\begin{theorem}
The measure $\nu_{\varepsilon}$ converges weakly to the measure
$\nu$ as $\varepsilon\rightarrow0$, see (\ref{e1}).
\end{theorem}

\begin{theorem}\label{t3}
Let $\nu_{\varepsilon}$ be the distribution of the limit
(\ref{e2}). We assume that (\ref{e3}), (\ref{e4}) hold, and
\begin{equation} \label{conditionoft3}
\sum_{i=1}^lp_i^2\frac{\lambda_{i,\max}}{\lambda_{i,\min}^2}<1.
\end{equation}
Then for every sufficiently small $\varepsilon>0$,
we have that $\nu_{\varepsilon}\ll\mathcal{L}_1$ with density in
$L^2$. For the $L^2$-norm of the density we have the following
estimate
\[
  \| \nu_\varepsilon \|_2 \leq \frac{C_{\varepsilon}' }{\sqrt{\varepsilon}},
\]
where
\[
  C_{\varepsilon}' = \sqrt{ \frac{32}{\left( 1 - \sum_{i=1}^l p_i^2
  \frac{(1 + \varepsilon) \lambda_{i,\max}}{((1 - \varepsilon)
  \lambda_{i,\min} )^2} \right) C_{\varepsilon}''} }
\]
and
\[
  C_{\varepsilon}'' = \min_{i \neq j}
  \left\{\frac{|a_i - a_j| + \varepsilon(-|a_i + a_j| - 2)}{1 - \varepsilon^2} \right\}.
\]
\end{theorem}

We can draft an easy corollary of the theorem.

\begin{cor}\label{c1}
  Let
  $\left\{\lambda_iY_{\varepsilon}x+a_i(1-\lambda_iY_{\varepsilon})\right\}_{i=1}^l$
  be a random iterated function system. We assume that (\ref{e4})
  holds, and
  \begin{equation}
    \sum_{i=1}^l\frac{p_i^2}{\lambda_i}<1.
  \end{equation}
  Then for every sufficiently small $\varepsilon >0$, we have that
  $\nu_{\varepsilon} \ll \mathcal{L}_1$ with density in $L^2$, the $L^2$-norm of the density satisfies
\[
  \| \nu_\varepsilon \|_2 \leq \frac{C_{\varepsilon}' }{\sqrt{\varepsilon}},
\]
where
\[
  C_{\varepsilon}' = \sqrt{ \frac{32}{\left( 1 - \sum_{i=1}^l p_i^2
  \frac{1 + \varepsilon}{(1 - \varepsilon)^2
  \lambda_{i}} \right) C_{\varepsilon}''} }
\]
and
\[
  C_{\varepsilon}'' = \min_{i \neq j}
  \left\{\frac{|a_i - a_j| + \varepsilon(-|a_i + a_j| - 2)}{1 - \varepsilon^2} \right\}.
\]
\end{cor}

We study an other case of random perturbation, namely let
$\widetilde{\lambda}_{i,\varepsilon}$ be uniformly distributed on
$[\lambda_i-\varepsilon,\lambda_i+\varepsilon]$. Let
$\left\{\widetilde{\lambda}_{i,\varepsilon}x+a_i(1-\widetilde{\lambda}_{i,\varepsilon})\right\}_{i=1}^l$ be
our random iterated function system, where $a_i\neq a_j$  for
every $i\neq j$. Let
$\underline{\lambda}=(\lambda_1,\ldots,\lambda_l)$, and
$X_{\underline{\lambda},\varepsilon}$ be the following random
variable
\begin{equation}
  X_{\underline{\lambda}, \varepsilon} = \sum_{k = 1}^{\infty} (a_{i_k} (1 - \widetilde{\lambda}_{i_k, \varepsilon})) \prod_{j = 1}^{k - 1} \widetilde{\lambda}_{i_j, \varepsilon}
\end{equation}
where the numbers $i_k$ are i.i.d., with the distribution $\mu$ on
$\left\{1,\ldots,l\right\}$, and
$\widetilde{\lambda}_{i_k,\varepsilon}$ are pairwise independent.
Let $\nu_{\underline{\lambda},\varepsilon}$ denote the
distribution of the random variable
$X_{\underline{\lambda},\varepsilon}$. Moreover let
$\nu_{\underline{\lambda}}$ be the invariant measure of
the the iterated function system $\left\{\lambda_ix+a_i(1-\lambda_i)\right\}_{i=1}^l$ according to
$\mu$.

\begin{theorem}
  The measure $\nu_{\underline{\lambda},\varepsilon}$ converges
  weakly to the measure $\nu_{\underline{\lambda}}$ as
  $\varepsilon\rightarrow0$.
\end{theorem}

To have a similar statement as in Theorem \ref{t3} we need a
technical assumption, namely
\begin{equation}\label{a1}
  \min_{i \neq j}
  \left|\frac{a_j\lambda_i - a_i\lambda_j}{\lambda_i - \lambda_j}\right| > 1.
\end{equation}

\begin{theorem}\label{t5}
  Let us suppose that (\ref{a1}) and (\ref{e4}) hold, and moreover that
  \begin{equation}
    \sum_{i = 1}^l \frac{p_i^2}{\lambda_i} < 1.
  \end{equation}
  Then for every sufficiently small $\varepsilon>0$, the measure
  $\nu_{\underline{\lambda},\varepsilon}$ is absolutely continuous
  with density in $L^2$, and the $L^2$-norm of the density satisfies
  \[
    \| \nu_{\underline{\lambda},\varepsilon} \|_2 \leq \frac{C_{\varepsilon}' }{\sqrt{\varepsilon}},
  \]
  where
  \[
    C_{\varepsilon}' = \sqrt{ \frac{32}{\left( 1 - \sum_{i=1}^l p_i^2
    \frac{\lambda_{i} + \varepsilon }{( \lambda_{i} - \varepsilon
     )^2} \right) C_{\varepsilon}''} }
  \]
  and
  \[
    C_{\varepsilon}'' = \sigma \min_{i \neq j}
    \frac{ |a_i \lambda_j - a_j \lambda_i |-|\lambda_i-\lambda_j|}{ \lambda_i \lambda_j }.
  \]
  where $0 < \sigma < 1$.
\end{theorem}

The main difference between Theorem \ref{t5} and Corollary
\ref{c1} is the random perturbation. Namely, in Theorem \ref{t5}
we choose the contracting ratio uniformly in the $\varepsilon$
neighborhood of $\lambda_i$, but in Corollary \ref{c1} we choose
the contraction ratio uniformly in the $\lambda_i\varepsilon$
neighborhood of $\lambda_i$.

Throughout this paper we will use the method in \cite{Pers}.

\section{Proof of Theorem \ref{t3}} \label{sec:firstproof}

Let $Q=[-1,1)^3$ and $m\in\N$. We partition the cube $Q$ into the
rectangles $\left\{Q_{1,k},\ldots,Q_{l,k}\right\}_{k=0}^{2^m-1}$,
where
\begin{multline*}
 Q_{i,k} = \biggl\{\, (x,y,z) \in Q : -1 + 2 \sum_{j=1}^{i-1} p_j \leq
 y < -1 + 2 \sum_{j=1}^i p_i, \\
 -1 + k 2^{-m+1} \leq
 z < -1 + (k + 1) 2^{-m+1} \,\biggr\},
\end{multline*}
where we use the convention that an empty sum is $0$. Hence we slice $Q$ in $2^m$ slices along the $z$-axis and $l$ slices along the $y$-axis. We thereby get $2^m l$ pieces which we call $Q_{i,k}$, according to the definition above.

Let
\[
  Q_i = \bigcup_{k = 0}^{2^m - 1} Q_{i,k}.
\]
For $(x, y, z) \in Q_i$, define $g_{\varepsilon,m} \colon Q \to Q$ by
\[
  g_{\varepsilon,m} \colon (x,y,z) \mapsto
  \left( d(z) f_i(x) + a_i (1 - d(z)),\ \frac{1}{p_i} y + b(y),\ 2^m z + c(z) \right),
\]
where
\begin{align*}
  d(z) &= 1 + 2^m \varepsilon (z - ( - 1 + (k + \frac{1}{2} ) 2^{-m+1} ), &\mathrm{for}\ (x,y,z) \in Q_{i,k}, \\
  b(y) &= 1 - \frac{1}{p_i} \left( - 1 + 2 \sum_{j=1}^i p_j \right), &\mathrm{for} \ (x,y,z) \in Q_{i,k}, \\
  c(z) &= 2^m - 2k -1, &\mathrm{for}\ (x,y,z) \in Q_{i,k}.
\end{align*}
Hence $g_{\varepsilon, m}$ maps each of the pieces $Q_{i, j}$ so that it s contracted in the $x$-direction and fully expanded in the $y$- and $z$-directions.

Let $\mathcal{L}_3$ be the normalised Lebesgue measure on $Q$. The
measures
\[
\gamma_{\varepsilon, m,
n}=\frac{1}{n}\sum_{k=0}^{n-1}\mathcal{L}_3\circ
g^{-k}_{\varepsilon,m}
\]
converge weakly to an SRB-measure $\gamma_{\varepsilon, m}$ as
$n\rightarrow\infty$. The measure $\gamma_{\varepsilon, m}$ is
clearly ergodic. Moreover, let $\nu_{\varepsilon, m}$ be the
projection of $\gamma_{\varepsilon, m}$ onto the first coordinate.
More precisely, if $E\subset[-1,1)$ is a measurable set, then
$\nu_{\varepsilon, m}(E)=\gamma_{\varepsilon,
m}(E\times[-1,1)\times[-1,1))$.

The measure $\nu_{\varepsilon, m}$ is the distribution of the limit
\[
  \lim_{n \rightarrow \infty} f_{i_1, y_{1,\varepsilon}} \circ
  f_{i_2, y_{2,\varepsilon}} \circ \cdots \circ
  f_{i_n, y_{n,\varepsilon}} (0),
\]
where $y_{i,\varepsilon}$ are uniformly distributed on $[1 -
\varepsilon, 1 + \varepsilon]$, but not independent. However, one
can easily prove the following lemma.
\begin{lemma}
The measure $\nu_{\varepsilon, m}$ converges weakly to
$\nu_{\varepsilon}$ as $m \rightarrow \infty$.\end{lemma}

Let
\[
A_i=\left\{(i,0),(i,1),\ldots,(i,2^m-1)\right\}
\]
and
\[
A=\bigcup_{i=1}^lA_i.
\]
Let $\Theta_0 = A^{ \N \cup \{0\} }$. If $p\in Q$ then there is a unique sequence
$\rho_0(p)=\left\{\rho_0(p)_k\right\}_{k=0}^{\infty}\in\Theta_0$
such that
\[
g_{\varepsilon,m}^k (p) \in Q_{\rho_0(p)_k},\ k = 0, 1, \ldots
\]
The map $\rho_0 \colon Q \to \Theta_0$ is not injective.

We can transfer the measures $\gamma_{\varepsilon, m}$ to a
measure $\gamma_{\Theta_0}$ by
$\gamma_{\Theta_0}=\gamma_{\varepsilon, m}\circ\rho_0^{-1}$.

We let $\Theta$ denote the natural extension of $\Theta_0$. That
is, $\Theta$ is the set of all two sides infinite sequences such
that any one sided infinite subsequence of sequence in $\Theta$ is
a sequence in $\Theta_0$. The measures $\gamma_{\Theta_0}$ defines
an ergodic measure $\gamma_{\Theta}$ on $\Theta$ in a natural way.
If $\xi \colon \Theta \to \Theta_0$ is defined by $\xi ( \left\{
i_k \right\}_{k \in \Z} ) = \left\{ i_k \right\}_{k \in \N \cup
\{0\} }$, then $\gamma_{\Theta_0}(E)=\gamma_{\Theta}(\xi^{-1}E)$.
We can define a map $\rho^{-1} \colon \Theta \to Q$ such that
$\rho^{-1} (\sigma(\boldsymbol{a})) = g_{\varepsilon, m}
(\rho^{-1} (\boldsymbol{a}))$ holds for any sequence
$\boldsymbol{a}\in\Theta$.

We note that the $L^2$ norm of the density $\nu_{\varepsilon, m}$
is not larger than twice that of the density of $\gamma_{\varepsilon,
m}$. If $h_{\nu_{\varepsilon, m}}(x)$ and $h_{\gamma_{\varepsilon,
m}}(x,y,z)$ denote the density of $\nu_{\varepsilon, m}$ and
$\gamma_{\varepsilon, m}$ respectively, then by Lyapunov's
inequality
\begin{align*}
  \|\nu_{\varepsilon, m}\|_2^2
  & \leq\int_{-1}^1 h_{\nu_{\varepsilon, m}} (x)^2 \, dx
  = 32\int_{-1}^1\left(\int_{-1}^1\int_{-1}^1h_{\gamma_{\varepsilon, m}} (x,y,z) \, \frac{dy}{2} \frac{dz}{2} \right)^2 \, \frac{dx}{2} \\
  & \leq 32 \int_{-1}^1 \int_{-1}^1 \int_{-1}^1
  h_{\gamma_{\varepsilon, m}} (x,y,z)^2 \, \frac{dy}{2} \frac{dz}{2} \frac{dx}{2}
  = 4 \|\gamma_{\varepsilon, m}\|_2^2.
\end{align*}
This proves that if $\gamma_{\varepsilon, m}$ has $L^2$ density,
then so has $\nu_{\varepsilon, m}$, and
\begin{equation}\label{e14}
\|\nu_{\varepsilon, m}\|_2\leq2\|\gamma_{\varepsilon, m}\|_2.
\end{equation}

\begin{lemma}\label{l1}
  Let
  \[
   C_p = \left\{\, (u,v,w) \in
   T_p Q : \Bigl| \frac{u}{w} \Bigr|,
   \Bigl| \frac{v}{w} \Bigr| < \frac{2^{m+1} \varepsilon}{2^m - \lambda_{\max,\max}
   (1 + \varepsilon)} \right\},
  \]
  where $p\in Q$ and
  $\lambda_{\max,\max} = \max_i \lambda_{i,\max} = \max_i \sup_{x \in [-1,1)} | f_i'(x) |$.
  The cones $C_p$ defines a family of unstable cones, that is
  $d_pg_{\varepsilon,m} (C_p) \subset C_{g_{\varepsilon,m} (p)}$.

  Moreover, for sufficiently large $m$ and every
  $0 < \varepsilon < \min_{i \neq j} \frac{|a_i - a_j|}{2 + |a_i + a_j|}$, if
  $\zeta_1 \subset Q_{\xi_1}$ and $\zeta_2 \subset Q_{\xi_2}$ are two
  curves segments with tangents in $C_p$ such that $\xi_1 \in A_i$ and
  $\xi_2 \in A_j$, $i \neq j$, then if $g_{\varepsilon, m} (\zeta_1)$
  and $g_{\varepsilon, m} (\zeta_2)$ intersects, and if $(u_1, v_1, 1)$
  and $(u_2, v_2, 1)$ are tangents to $g_{\varepsilon, m} (\zeta_1)$ and
  $g_{\varepsilon, m} (\zeta_2)$ respectively, it holds
  $|u_1 - u_2| > C_{\varepsilon, m} \varepsilon$, where
  \[
    C_{\varepsilon, m} = \min_{i \neq j}
    \left\{ \frac{|a_i - a_j| + \varepsilon (-|a_i + a_j| - 2)}{1 - \varepsilon^2} - \frac{4(1 + \varepsilon) \lambda_{\max, \max}}{2^m - \lambda_{\max, \max}(1 + \varepsilon)} \right\}.
  \]
\end{lemma}

\begin{proof}[Proof of Lemma \ref{l1}]
The Jacobian of $g_{\varepsilon,m}$ is
\[
  d_pg_{\varepsilon,m}=\left(%
  \begin{array}{ccc}
   d(z)f_i'(x) & 0 & 2^m\varepsilon(f_i(x)-a_i) \\
   0 & \frac{1}{p_i} & 0 \\
   0 & 0 & 2^m \\
  \end{array}%
  \right)
\]
where $p=(x,y,z)\in Q_{i,k}$. If $(u,v,w)\in C_p$, then
\[
d_pg_{\varepsilon,m}(u,v,w)=\left(%
\begin{array}{c}
  d(z)f_i'(x)u+2^m\varepsilon(f_i(x)-a_i)w \\
  \frac{1}{p_i}v \\
  2^mw \\
\end{array}%
\right)
\]
The estimates
\begin{multline*}
  \frac{|d(z) f_i'(x) u + 2^m \varepsilon (f_i(x) - a_i) w|}{|2^m w|}
  \leq \frac{(1 + \varepsilon) \lambda_{i, \max}}{2^m} \frac{|u|}{|w|} + 2 \varepsilon\\
  \leq \frac{(1 + \varepsilon) \lambda_{i, \max}}{2^m} \frac{2^{m+1} \varepsilon}{2^m - (1 + \varepsilon) \lambda_{\max, \max}} + 2 \varepsilon \leq \frac{2^{m+1} \varepsilon}{2^m - (1 + \varepsilon) \lambda_{\max, \max}}
\end{multline*}
and
\[
  \frac{|\frac{1}{p_i} v|}{|2^m w|} \leq \frac{1}{p_i 2^m} \frac{2^{m+1} \varepsilon}{2^m - (1 + \varepsilon) \lambda_{\max, \max}} \leq \frac{2^{m+1} \varepsilon}{2^m - (1 + \varepsilon) \lambda_{\max, \max}}
\]
proves that $d_pg_{\varepsilon, m} (C_p) \subset
C_{g_{\varepsilon,m} (p)}$ if $m$ is sufficiently large, so that
$2^m - (1+ \varepsilon) \lambda_{\max, \max} > 0$ and $p_i2^m>1$.

To prove the other statement of the Lemma, assume that
$p=(x_p,y_p,z_p)\in Q_i$ and $q=(x_q,y_q,z_q)\in Q_j$, $i\neq j$,
are such that $g_{\varepsilon,m}(p)=g_{\varepsilon,m}(q)=(x,y,z)$.
Then, if $p \in Q_i$
\[
  d_pg_{\varepsilon, m} \colon (u,v,1) \mapsto 2^m \left( \frac{d(z_p) f_i'(x_p)}{2^m} u + (f_i (x_p) - a_i) \varepsilon,\ \frac{v}{p_i},\ 1 \right)
\]
Then
\[
f_i(x_p)=\frac{x-a_i(1-d(z_p))}{d(z_p)} \quad \mathrm{and} \quad f_j(x_q)=\frac{x-a_j(1-d(z_q))}{d(z_q)}.
\]
Without loss of generality, let us assume that $a_i>a_j$. For
simplicity we study the case $x\geq a_i>a_j$. The proof of the
other cases $a_i\geq x\geq a_j$ and $a_i>a_j\geq x$ is similar.
Then
\begin{multline*}
  d_p g_{\varepsilon,m} (C_p)
  \subset \biggl\{\, w(u,v,1) :
  \frac{x - a_i}{1 + \varepsilon} \varepsilon - \frac{2(1 + \varepsilon)
  \lambda_{i, \max} \varepsilon}{2^m - \lambda_{\max, \max}(1 + \varepsilon)} \\
  \leq u \leq \frac{x - a_i}{1 - \varepsilon} \varepsilon +
  \frac{2(1 + \varepsilon) \lambda_{i, \max} \varepsilon}{2^m - \lambda_{\max, \max}(1 + \varepsilon)}
  \biggr\}
\end{multline*}
Therefore
\begin{align*}
  |u_2 - u_1| &\geq \frac{x - a_j}{1 + \varepsilon} \varepsilon - \frac{x - a_i}{1 - \varepsilon} \varepsilon - \frac{2(1 + \varepsilon) (\lambda_{i, \max} + \lambda_{j, \max}) \varepsilon}{2^m - \lambda_{\max, \max}(1 + \varepsilon)}\\
  &\geq \left(\frac{a_i - a_j + \varepsilon(a_i + a_j - 2)}{1 - \varepsilon^2} - \frac{4(1 + \varepsilon) \lambda_{\max, \max}}{2^m - \lambda_{\max, \max}(1 + \varepsilon)} \right) \varepsilon
\end{align*}
for every $x\geq a_i>a_j$. Since $0<\varepsilon<\min_{i\neq
j}\frac{|a_i-a_j|}{2+|a_i+a_j|}$,
\[
\frac{a_i-a_j+\varepsilon(a_i+a_j-2)}{1-\varepsilon^2}>0
\]
Therefore
\[
\frac{a_i-a_j+\varepsilon(a_i+a_j-2)}{1-\varepsilon^2}-\frac{4(1+\varepsilon)\lambda_{\max,\max}}{2^m-\lambda_{\max,\max}(1+\varepsilon)}>0
\]
for sufficiently large $m$. By similar methods, we have for
$a_i\geq x\geq a_j$
\[
|u_2-u_1|\geq\left(\frac{a_i-a_j}{1+\varepsilon}-\frac{4(1+\varepsilon)\lambda_{\max,\max}}{2^m-\lambda_{\max,\max}(1+\varepsilon)}\right)\varepsilon
\]
and for $a_i>a_j\geq x$
\[
|u_2-u_1|\geq\left(\frac{a_i-a_j-\varepsilon(a_i+a_j+2)}{1-\varepsilon^2}-\frac{4(1+\varepsilon)\lambda_{\max,\max}}{2^m-\lambda_{\max,\max}(1+\varepsilon)}\right)\varepsilon
\]
Therefore we can choose $C_{\varepsilon,m}$ as
\[
C_{\varepsilon,m}=\min_{i\neq
j}\left\{\frac{|a_i-a_j|+\varepsilon(-|a_i+a_j|-2)}{1-\varepsilon^2}-\frac{4(1+\varepsilon)\lambda_{\max,\max}}{2^m-\lambda_{\max,\max}(1+\varepsilon)}\right\}. \qedhere
\]
\end{proof}

The rest of the proof follows Tsujii's article \cite{tsujii}.

\begin{proof}[Proof of Theorem \ref{t3}]
For any $r>0$ we define the bilinear form $(\cdot, \cdot)_r$ of signed
measures on $\R$ by
\[
(\rho_1,\rho_2)_r=\int_{\R}\rho_1(B_r(x))\rho_2(B_r(x)) \, dx
\]
where $B_r(x)=[x-r,x+r]$. It is easy to see that if
\[
\liminf_{r\rightarrow0}\frac{1}{r^2}(\rho,\rho)_r<\infty
\]
then the measure $\rho$ has density in $L^2$, moreover
\[
\|\rho\|_2^2\leq\liminf_{r\rightarrow0}\frac{1}{r^2}(\rho,\rho)_r.
\]
Let $\gamma_{z}$ denote the conditional measure of
$\gamma_{\varepsilon,m}$ on the set $R_{z} = \{\, (u,v,w) \in
Q : v = y, w = z \,\}$. Note that $\gamma_z$ is independent of $y$
almost everywhere. Let
\[
  J(r):=\frac{1}{r^2}\int_{-1}^1(\gamma_{z},\gamma_{z})_r \, dz.
\]
It is easy to see that
\begin{equation}\label{e5}
  \|\gamma_{\varepsilon,m} \|_2^2 = \int_{-1}^1 \|\gamma_{z}\|_2^2 \, dz.
\end{equation}
By the invariance of $\gamma_{\varepsilon,m}$ it follows that
\begin{equation}\label{e6}
  \gamma_z = 2^{-m}\sum_{i=1}^lp_i \sum_{a \in A_i}
  \gamma_{g_{\varepsilon,m}^{-a} (z)} \circ g_{\varepsilon, m}^{-a},
\end{equation}
where $g_{\varepsilon,m}^{-a}$ denotes the inverse branch of
$g_{\varepsilon,m}$ such that the image of $g_{\varepsilon,m}^{-a}$ is in the cylinder $[a]$. Then by
(\ref{e6}) and the definition of $J(r)$
\begin{equation}\label{e8}
  J(r) = \frac{1}{2^{2m} r^2} \sum_{i=1}^l \sum_{j=1}^l p_i p_j
  \sum_{a \in A_i} \sum_{b \in A_j} \int_{-1}^1
  (\gamma_{g_{\varepsilon, m}^{-a}(z)} \circ
  g_{\varepsilon, m}^{-a}, \gamma_{g_{\varepsilon,m}^{-b}(z)} \circ
  g_{\varepsilon, m}^{-b})_r \, dz.
\end{equation}
For fixed $a,b\in A_i$ it holds,
\begin{align}
  (& \gamma_{g_{\varepsilon,m}^{-a}(z)} \circ g_{\varepsilon, m}^{-a}, \gamma_{g_{\varepsilon,m}^{-b}(z)} \circ g_{\varepsilon, m}^{-b})_r \nonumber \\
  &\leq(\gamma_{g_{\varepsilon,m}^{-a}(z)}\circ
  g_{\varepsilon, m}^{-a},\gamma_{g_{\varepsilon,m}^{-a}(z)}\circ
  g_{\varepsilon, m}^{-a})_r^{\frac{1}{2}} (\gamma_{g_{\varepsilon,m}^{-b}(z)} \circ
  g_{\varepsilon, m}^{-b}, \gamma_{g_{\varepsilon, m}^{-b}(z)} \circ
  g_{\varepsilon, m}^{-b})_r^{\frac{1}{2}} \nonumber \\
  &\leq (1 + \varepsilon) \lambda_{i,\max} (\gamma_{g_{\varepsilon, m}^{-a} (z)}, \gamma_{g_{\varepsilon, m}^{-a} (z)})_{\frac{r}{(1-\varepsilon) \lambda_{i,\min}}}^{\frac{1}{2}} \times (\gamma_{g_{\varepsilon, m}^{-b} (z)}, \gamma_{g_{\varepsilon, m}^{-b} (z)})_{\frac{r}{(1-\varepsilon) \lambda_{i,\min}}}^{\frac{1}{2}} \nonumber \\
  &\leq (1 + \varepsilon) \lambda_{i,\max} \frac{(\gamma_{g_{\varepsilon, m}^{-a} (z)}, \gamma_{g_{\varepsilon, m}^{-a} (z)})_{\frac{r}{(1 - \varepsilon) \lambda_{i, \min}}} + (\gamma_{g_{\varepsilon, m}^{-b} (z)}, \gamma_{g_{\varepsilon, m}^{-b} (z)})_{\frac{r}{(1 - \varepsilon) \lambda_{i, \min}}}}{2}. \label{e9}
\end{align}
Moreover, if $a\in A_i$ and $b\in A_j$, $i\neq j$, then
\begin{align}
  (\gamma_{g_{\varepsilon,m}^{-a}(z)} &\circ g_{\varepsilon,
  m}^{-a},\gamma_{g_{\varepsilon,m}^{-b}(z)}\circ g_{\varepsilon,
  m}^{-b})_r \nonumber \\
  &=\int\gamma_{g_{\varepsilon,m}^{-a}(z)}\circ
  g_{\varepsilon,
  m}^{-a}(B_r(x))\gamma_{g_{\varepsilon,m}^{-b}(z)}\circ
  g_{\varepsilon,
  m}^{-b}(B_r(x))dx \nonumber \\
  &=\int \!\! \int \!\! \int \mathbb{I}_{ \left\{\, s : |s-x| < r \, \right\} } (s) \mathbb{I}_{ \left\{\, t : |t-x| < r \,\right\} } (t) \nonumber \\
  & \hspace{3cm} d\gamma_{g_{\varepsilon, m}^{-a} (z)} \circ
  g_{\varepsilon, m}^{-a} (s) d\gamma_{g_{\varepsilon,m}^{-b}(z)} \circ
  g_{\varepsilon, m}^{-b} (t) dx \nonumber \\
  &\leq \int \!\! \int 2r \mathbb{I}_{ \left\{\, (s,t) : |s-t| < 2r \,\right\}} (s,t) \, d\gamma_{g_{\varepsilon, m}^{-a} (z)} \circ
  g_{\varepsilon, m}^{-a} (s) d\gamma_{g_{\varepsilon,m}^{-b} (z)} \circ
  g_{\varepsilon, m}^{-b} (t) \nonumber \\
  &= \int \!\! \int \mathbb{I}_{\left\{\, (\boldsymbol{c}, \boldsymbol{d}) : |\rho^{-1} (\cdots c_{-2} c_{-1} a \rho_0(z) ) - \rho^{-1} (d_{-2} d_{-1} b \rho_0 (z)) | < 2r \,\right\}} (\boldsymbol{c}, \boldsymbol{d}) \nonumber \\
  &\hspace{8cm} d\gamma_{\Theta} (\boldsymbol{c}) d\gamma_{\Theta} (\boldsymbol{d}). \label{e7}
\end{align}
Therefore by Lemma \ref{l1} and (\ref{e7}) we get that
\begin{align}
  \int_{-1}^1 &(\gamma_{g_{\varepsilon,m}^{-a}(z)} \circ
  g_{\varepsilon, m}^{-a},\gamma_{g_{\varepsilon,m}^{-b}(z)}\circ
  g_{\varepsilon, m}^{-b})_rdz \nonumber \\
  &\leq 2r \int \!\! \int \mathcal{L}_1 (\left \{\, z : |\rho^{-1} (\cdots
  c_{-2} c_{-1} a \rho_0(z)) - \rho^{-1} (d_{-2} d_{-1} b \rho_0 (z)) | < 2r \,\right\}) \nonumber \\
  &\hspace{9.5cm} d\gamma_{\Theta}(\boldsymbol{c})d\gamma_{\Theta}(\boldsymbol{d}) \nonumber \\
  &\leq \frac{8r^2}{C_{\varepsilon,m} \varepsilon}. \label{e10}
\end{align}
Then by using (\ref{e8}) we have
\begin{multline*}
  J(r)=\frac{1}{2^{2m}r^2}\sum_{i=1}^lp_i^2\sum_{a,b\in
  A_i}\int_{-1}^1(\gamma_{g_{\varepsilon,m}^{-a}(z)}\circ
  g_{\varepsilon, m}^{-a},\gamma_{g_{\varepsilon,m}^{-b}(z)}\circ
  g_{\varepsilon, m}^{-b})_r \, dz\\
  +\frac{1}{2^{2m}r^2}\sum_{i\neq j}p_ip_j\sum_{a\in A_i}\sum_{b\in
  A_j}\int_{-1}^1(\gamma_{g_{\varepsilon,m}^{-a}(z)}\circ
  g_{\varepsilon, m}^{-a},\gamma_{g_{\varepsilon,m}^{-b}(z)}\circ
  g_{\varepsilon, m}^{-b})_r \, dz.
\end{multline*}
Then we can give an upper bound for the first part of the sum using
(\ref{e9}) and an integral transformation

\begin{align}
&\frac{1}{2^{2m}r^2}\sum_{i=1}^lp_i^2\sum_{a,b\in
A_i}\int_{-1}^1(\gamma_{g_{\varepsilon,m}^{-a}(z)}\circ
g_{\varepsilon, m}^{-a},\gamma_{g_{\varepsilon,m}^{-b}(z)}\circ
g_{\varepsilon,
m}^{-b})_r \, dz \nonumber \\
&\leq\frac{1}{2^{2m}r^2}\sum_{i=1}^lp_i^2(1+\varepsilon)\lambda_{i,\max}2^m\sum_{a\in
A_i}\int_{-1}^1(\gamma_{g_{\varepsilon,m}^{-a}(z)},\gamma_{g_{\varepsilon,m}^{-a}(z)})_{\frac{r}{(1-\varepsilon)\lambda_{i,\min}}} \, dz \nonumber \\
&\leq\frac{1}{2^{2m}r^2}\sum_{i=1}^lp_i^2(1+\varepsilon)\lambda_{i,\max}2^m\sum_{k=0}^{2^m-1}2^m\int_{-1+k2^{-m+1}}^{-1+(k+1)2^{-m+1}}(\gamma_{z},\gamma_{z})_{\frac{r}{(1-\varepsilon)\lambda_{i,\min}}} \, dz \nonumber \\
&\leq\sum_{i=1}^lp_i^2\frac{(1+\varepsilon)\lambda_{i,\max}}{((1-\varepsilon)\lambda_{i,\min})^2}\frac{1}{\left(\dfrac{r}{(1-\varepsilon)\lambda_{i,\min}}\right)^2}\int_{-1}^1(\gamma_{z},\gamma_{z})_{\frac{r}{(1-\varepsilon)\lambda_{i,\min}}} \, dz \nonumber \\
&\leq\max_{i}J\left(\frac{r}{\lambda_{i,\min}(1-\varepsilon)}\right)\sum_{i=1}^lp_i^2\frac{(1+\varepsilon)\lambda_{i,\max}}{((1-\varepsilon)\lambda_{i,\min})^2}. \label{e11}
\end{align}
For the second part of the sum, we use (\ref{e10}), to prove that it is bounded by
\begin{multline}\label{e12}
\frac{1}{2^{2m}r^2}\sum_{i\neq j}p_ip_j\sum_{a\in A_i}\sum_{b\in
A_j}\int_{-1}^1(\gamma_{g_{\varepsilon,m}^{-a}(z)}\circ
g_{\varepsilon, m}^{-a},\gamma_{g_{\varepsilon,m}^{-b}(z)}\circ
g_{\varepsilon, m}^{-b})_r \, dz\\
\leq \frac{1}{2^{2m}r^2}\sum_{i\neq j}p_ip_j\sum_{a\in
A_i}\sum_{b\in
A_j}\frac{8r^2}{C_{\varepsilon,m}\varepsilon}\leq\frac{8}{C_{\varepsilon,m}\varepsilon}.
\end{multline}
By using (\ref{e11}) and (\ref{e12}) we have
\begin{equation}\label{e13}
J(r)\leq\frac{8}{C_{\varepsilon,m}\varepsilon}+b\max_{i}J\left(\frac{r}{\lambda_{i,\min}(1-\varepsilon)}\right)
\end{equation}
where
$b=\sum_{i=1}^lp_i^2\frac{(1+\varepsilon)\lambda_{i,\max}}{((1-\varepsilon)\lambda_{i,\min})^2}$
is less than $1$ by \eqref{conditionoft3}. for sufficiently small
$\varepsilon>0$. We define a strictly monotone decreasing series
$r_k$. Let $r_0<1/2$ be fixed and
$r_k=r_0(1-\varepsilon)^k\prod_{n=1}^k(\lambda_{i_n,\min})$ such
that
\[
\max_{i}J\left(\frac{r_k}{(1-\varepsilon)\lambda_{i,\min}}\right)=J(r_{k-1}).
\]
We note that $r_k$ is a well defined series. Then by induction and
by using (\ref{e13}), we have
\begin{equation}\label{e15}
J(r_k)\leq\frac{8}{C_{\varepsilon,m}\varepsilon}\frac{1-b^{k}}{1-b}+b^kJ(r_0)
\end{equation}
for every $k\geq1$. Hence by (\ref{e14}), (\ref{e5}) and
(\ref{e15}) we get
\begin{multline}
\|\nu_{\varepsilon,
m}\|_2^2\leq4\liminf_{r\rightarrow0}J(r)\leq4\liminf_{k\rightarrow\infty}J(r_k)
\\ \leq \frac{32}{C_{\varepsilon,m} \varepsilon} \frac{1}{1 -
\sum_{i=1}^l p_i^2 \frac{(1 + \varepsilon) \lambda_{i,\max}}{((1 -
\varepsilon) \lambda_{i, \min})^2}}.
\end{multline}
Since $\nu_{\varepsilon,m}$ converges weakly to $\nu_\varepsilon$ we get that
\begin{equation}
\|\nu_{\varepsilon}\|_2\leq\frac{1}{\sqrt{\varepsilon}}C_{\varepsilon}'
\end{equation}
where
\[
C_{\varepsilon}'=\sqrt{\frac{32}{\left(1-\sum_{i=1}^lp_i^2\frac{(1+\varepsilon)\lambda_{i,\max}}{((1-\varepsilon)\lambda_{i,\min})^2}\right)C_{\varepsilon}''}}
\]
and
\[
C_{\varepsilon}'' = \lim_{m \to \infty} C_{\varepsilon, m} = \min_{i \neq j}
\left\{\frac{|a_i - a_j| + \varepsilon(-|a_i + a_j| - 2)}{1 - \varepsilon^2} \right\}.
\]
\end{proof}

\section{Proof of Theorem \ref{t5}}

We do not notify the proof of Theorem \ref{t5}, because it is
similar to the proof of Theorem \ref{t3}. We notify only the
modification of Lemma \ref{l1}, which is important as it proves
transversality.

First we define a new dynamical system. Let $Q_{i,k}$ and
$A_{i,k}$ be as in Section~\ref{sec:firstproof}. Let
$\widetilde{g}_{\varepsilon,m} \colon Q \to Q$ be defined by
\[
  \widetilde{g}_{\varepsilon,m} \colon  (x,y,z) \mapsto \left( \widetilde{d}(z) x + a_i (1 - \widetilde{d} (z)),\ \frac{1}{p_i} y + b (y),\ 2^m z + c (z) \right),
\]
for $(x, y, z) \in Q_i$, where
\begin{align*}
  \widetilde{d}(z) &= \lambda_i + 2^m \varepsilon (z - ( - 1 + (k + \frac{1}{2}) 2^{-m+1})),
  &\mathrm{for}\ (x,y,z) \in Q_{i,k}, \\
  b(y) &= 1 - \frac{1}{p_i} \left( - 1 + 2 \sum_{j=1}^i p_j \right),
  &\mathrm{for}\ (x,y,z) \in Q_{i,k}, \\
  c(z) &= 2^m - 2k - 1, &\mathrm{for}\ (x,y,z) \in Q_{i,k}.
\end{align*}

\begin{lemma}\label{l2}
Let us suppose that (\ref{a1}) holds. Let
\[
 C_p = \left\{\, (u,v,w) \in T_p Q : \Bigl| \frac{u}{w} \Bigr|, \Bigl| \frac{v}{w} \Bigr| < \frac{2^{m+1} \varepsilon}{2^m - \lambda_{\max} - \varepsilon} \,\right\},
\]
where $p \in Q$ and $\lambda_{\max} = \max_i \lambda_{i}$. The cones
$C_p$ defines a family of unstable cones, that is
$d_p \widetilde{g}_{\varepsilon,m} (C_p) \subset
C_{\widetilde{g}_{\varepsilon,m} (p)}$.

Moreover, for sufficiently large $m$ and every sufficiently small
$0<\varepsilon$, if $\zeta_1\subset Q_{\xi_1}$ and $\zeta_2\subset
Q_{\xi_2}$ are two line segments with tangents in $C_p$ such that
$\xi_1\in A_i$ and $\xi_2\in A_j$, $i\neq j$, then if
$\widetilde{g}_{\varepsilon,m}(\zeta_1)$ and
$\widetilde{g}_{\varepsilon,m}(\zeta_2)$ intersects, and if
$(u_1,v_1,1)$ and $(u_2,v_2,1)$ are tangents to
$\widetilde{g}_{\varepsilon,m}(\zeta_1)$ and
$\widetilde{g}_{\varepsilon,m}(\zeta_2)$ respectively, there
exists a constant $C_{\varepsilon,m}$, depending on $\varepsilon$ and $m$, but bounded away from $0$ and infinity, such that
$|u_1-u_2|>C_{\varepsilon,m}\varepsilon$.
\end{lemma}

\begin{proof}[Proof of Lemma \ref{l2}]
The Jacobian of $\widetilde{g}_{\varepsilon,m}$
\[d_p\widetilde{g}_{\varepsilon,m}=\left(%
\begin{array}{ccc}
  \widetilde{d}(z) & 0 & 2^m\varepsilon(x-a_i) \\
  0 & \frac{1}{p_i} & 0 \\
  0 & 0 & 2^m \\
\end{array}%
\right),
\]
where $p=(x,y,z)\in Q_{i,k}$. If $(u,v,w) \in C_p$, then
\[
d_p\widetilde{g}_{\varepsilon,m}(u,v,w)=\left(%
\begin{array}{c}
  \widetilde{d}(z)u+2^m\varepsilon(x-a_i)w \\
  \frac{1}{p_i}v \\
  2^mw \\
\end{array}%
\right).
\]
The estimate
\begin{multline*}
\frac{|\widetilde{d}(z) u + 2^m \varepsilon (x - a_i) w|}{|2^m w|} \leq \frac{\widetilde{d} (z) |u|}{2^m|w|} + 2 \varepsilon \\ \leq \frac{\lambda_i + \varepsilon}{2^m} \frac{2^{m+1} \varepsilon}{2^m - \lambda_{\max} - \varepsilon} + 2\varepsilon \leq \frac{2^{m+1} \varepsilon}{2^m - \lambda_{\max} - \varepsilon}
\end{multline*}
shows that $d_p \widetilde{g}_{\varepsilon, m} (C_p) \subset C_{ \widetilde{g}_{\varepsilon, m} (p)}$.
Now we prove the other statement of the Lemma. Assume that
$p=(x_p,y_p,z_p)\in Q_i$ and $q=(x_q,y_q,z_q)\in Q_j$, $i\neq j$,
are such that
$\widetilde{g}_{\varepsilon,m}(p)=\widetilde{g}_{\varepsilon,m}(q)=(x,y,z)$.
Then
\[
p\in
Q_i\;\;\;\Rightarrow\;\;\;d_p\widetilde{g}_{\varepsilon,m}:(u,v,1)\mapsto2^m\left(\frac{\widetilde{d}(z_p)}{2^m}u+(x_p-a_i)\varepsilon,\ \frac{v}{p_i},\ 1\right).
\]
Then
\[
x_p=\frac{x - a_i (1 - \widetilde{d} (z_p))}{\widetilde{d} (z_p)}, \quad x_q = \frac{x-a_j(1 - \widetilde{d} (z_q))}{\widetilde{d} (z_q)}
\]
and
\begin{multline*}
  d_p\widetilde{g}_{\varepsilon,m}(C_p)
  \subset \biggl\{\, w(u,v,1):
  \frac{x - a_i}{\widetilde{d} (z_p)} \varepsilon - \frac{2 (\lambda_{i} + \varepsilon) \varepsilon}{2^m - \lambda_{\max} - \varepsilon} \\ \leq
  u \leq \frac{x - a_i}{\widetilde{d} (z_p)} \varepsilon + \frac{2(\lambda_i + \varepsilon) \varepsilon}{2^m - \lambda_{\max} - \varepsilon}
  \, \biggr\}.
\end{multline*}
Therefore
\[
  |u_2 - u_1| \geq \left( \left| \frac{x - a_i}{\widetilde{d} (z_p)} - \frac{x  - a_j}{\widetilde{d} (z_q)}\right| - \frac{2 (\lambda_{i} + \lambda_j + 2 \varepsilon)}{2^m - \lambda_{\max} - \varepsilon}\right) \varepsilon.
\]
The term
\[
  \left| \frac{x - a_i}{\widetilde{d} (z_p)} - \frac{x  - a_j}{\widetilde{d} (z_q)}\right|
\]
can be estimated by
\[
  \left| \frac{x - a_i}{\widetilde{d} (z_p)} - \frac{x  - a_j}{\widetilde{d} (z_q)}\right| \geq \left| \frac{ |\widetilde{d} (z_p) - \widetilde{d} (z_q) | |x| - | a_j \widetilde{d} (z_p) - a_i \widetilde{d} (z_q) | }{ \widetilde{d} (z_p) \widetilde{d} (z_q) } \right|.
\]
Hence, this term is positive provided that
\[
  | a_j \widetilde{d} (z_p) - a_i \widetilde{d} (z_q) | > |\widetilde{d} (z_p) - \widetilde{d} (z_q) |.
\]
Since
$\lambda_i - \varepsilon \leq \widetilde{d}(z_p) \leq \lambda_i + \varepsilon$
and
$\lambda_j-\varepsilon \leq \widetilde{d}(z_q)\leq\lambda_j+\varepsilon$,
this is implied by \eqref{a1} if $\varepsilon$ is sufficiently small.

If we let
\[
  C_{\varepsilon, m} = \frac{1}{2} \min_{i \ne j} \frac{ | a_i \lambda_j - a_j \lambda_i | - | \lambda_i - \lambda_j | }{\lambda_i \lambda_j},
\]
then
\[
  |u_2 - u_1| \geq C_{\varepsilon, m} \varepsilon,
\]
provided that $\varepsilon$ is small and $m$ large.

In fact we can let
\[
  C_{\varepsilon, m} = \sigma \min_{i \ne j} \frac{ | a_i \lambda_j - a_j \lambda_i | - | \lambda_i - \lambda_j | }{\lambda_i \lambda_j},
\]
for $0 < \sigma < 1$.
\end{proof}




\end{document}